\documentclass[reqno,b5paper]{amsart}
\usepackage{amsmath}
\usepackage{amssymb}
\usepackage{amsthm}
\usepackage{eqlist}
\usepackage{enumerate}
\usepackage[mathscr]{eucal}
\theoremstyle{plain}
\newtheorem{theorem}{Theorem}[section]
\newtheorem{prop}[theorem]{Proposition}

\theoremstyle{definition}
\newtheorem{definition}{Definition}[section]
\newtheorem{remark}{\textnormal{\textbf{Remark}}}
\newtheorem*{acknowledgement}{\textnormal{\textbf{Acknowledgement}}}

\theoremstyle{remark}

\numberwithin{equation}{section}

\begin{document}
\title[On a class of paracontact metric 3-manifolds]
{On a class of paracontact metric 3-manifolds}
\author[K. Srivastava \and S. K. Srivastava]
{K. Srivastava*\and S. K. Srivastava**}

\newcommand{\acr}{\newline\indent}

\address{\llap{*\,}Department of Mathematics\acr
                    D. D. U. Gorakhpur University\acr
                    Gorakhpur-273009\acr
                    Uttar Pradesh\acr
                    INDIA}
\email{ksriddu22@gmail.com}

\address{\llap{**\,}Department of Mathematics\acr
                   Central University of Himachal Pradesh\acr
                   Dharamshala - 176215\acr
                   Himachal Pradesh\acr
	       INDIA}
\email{sachink.ddumath@gmail.com}


\begin{abstract}
The purpose of this paper is to classify paracontact metric $3$-manifolds $M^3$ such that the Ricci operator $S$ commutes with the endomorhism $\phi$ of its tangent bundle $\Gamma(TM^3)$.
\end{abstract}

\maketitle

\section{Introduction}
Almost paracontact metric manifolds are the well known illustrations of almost para-CR manifolds. In \cite{skm} Kaneyuki and Kozai defined the almost paracontact structure on pseudo-Riemannian manifold $M^{2n+1}$ and constructed the almost paracomplex structure on $M^{2n+1}\times \mathbb{R}$. Analogous to Blair \emph{et al.} \cite{de1990} if the paracontact metric manifold $(M^{2n+1},\phi, \xi, \eta, g)$ is $\eta$-Einstein we do not have a widespread classification. This paper is organized as follows. In \S $2$, we present some technical apparatus which is needed for further investigations. In \S $3$, we first prove that for any $X$ orthogonal to $\xi$ the function $Trl$ vanishes and the function $f$ defined by $lX=fX$ is constant everywhere on a paracontact metric manifold $M^3$. We then show that the conditions, (i) the structure is $\eta$-Einstein, (ii) Ricci operator $S$ commutes with tensor field $\phi$  and (iii) $\xi$ belongs to the $k$-nullity distribution of pseudo-Riemannian metric $g$ are equivalent on $M^3$. Finally, we prove that the unit torse forming vector field in this manifold with $S\phi =\phi S$ is concircular in \S $4$. 
\section{Preliminaries}

A $C^{\infty}$ smooth manifold $M^{2n+1}$ of dimension $(2n+1)$, is said to have a triplet $(\phi, \xi, \eta)$-structure, if it admits a $(1, 1)$ tensor field $\phi$, a unique vector field $\xi$ called the characteristic vector field or Reeb vector field and a $1$-form $\eta$ satisfying: 
\begin{align}\label{eta}
\phi^2 = I -\eta\otimes \xi\,\,\, {\rm\and}\,\,\,\,  \eta\left(\xi\right)=1
\end{align}
where $I$ is the identity transformation. The endomorphism $\phi$ induces an almost paracomplex structure on each fibre of $D=ker(\eta)$, the contact subbundle that is the eigen distributions  $D^{\pm 1}$ corresponding to the characteristics values $\pm 1$ of $\phi$ have equal dimension $n$. \\
From the equation (\ref{eta}), it can be easily deduce that 
\begin{eqnarray}\label{phixi}
\phi\xi = 0, && \eta\,o\,\phi = 0 \,\,\,\,{\rm and \,\,\,\, rank}(\phi) = 2n.
\end{eqnarray}
This triplet structure $(\phi, \xi, \eta)$ is called an almost paracontact structure and the manifold $M^{2n+1}$ equipped with the $(\phi, \xi, \eta)$-structure is called an almost paracontact manifold \cite{skm}. If an almost paracontact manifold admits a pseudo-Riemannian metric \cite{sz}, $g$ satisfying: 
\begin{align}\label{gphi}
g\left(\phi X, \phi Y\right) = -g(X, Y) + \eta(X)\eta(Y)
\end{align}
where signature of $g$ is necessarily $(n+1, n)$ for any vector fields  $X$ and $Y$. Then the quadruple-$(\phi, \xi, \eta, g)$ is called an almost paracontact metric structure and the manifold $M^{2n+1}$ equipped with paracontact metric structure is called an almost paracontact metric manifold. With respect to $g$, $\eta$ is metrically dual to $\xi$, that is
\begin{align}\label{gx}
g(X, \xi) = \eta (X) 
\end{align}
Also, equation (\ref{gphi}) implies that
\begin{align}\label{gphix}
g(\phi X, Y) = -g(X, \phi Y).
\end{align}
{\it Note}: The above metric $g$ is, of course, not unique.\\
 Further, in addition to the above properties, if the quadruple-$(\phi, \xi, \eta, g)$ satisfies:\\
\begin{align}\label{detaxy}
d\eta(X, Y)= g(X, \phi Y),
\nonumber\end{align} 
for all vector fields $X,\,Y$ on $M^{2n+1}$, then the manifold is called a paracontact metric manifold and the corresponding structure-$(\phi, \xi, \eta, g)$ is called a paracontact structure with the associated metric $g$ \cite{sz}. 

Now, for an almost paracontact metric manifold, there always exists a special kind of local pseudo-orthonormal basis $\left\{X_{i}, X_{i^*}, \xi\right\}$;  where $X_{i^*}=\phi X_{i}$; $\xi$ and $X_{i}$'s are space-like vector fields and $X_{i^*}$'s  are time-like. Such a basis is called a $\phi$-basis. Indicating by $L$ and $R$, the Lie differentiation operator and the curvature tensor of $M^{2n+1}$ respectively, let us define, $(1,1)$ type tensor fields $h$ and $l$, which are symmetric as well, by $h= \frac{1}{2}L_{\xi}\phi,\,\,\, l=R(.,\,\xi)\xi$. The basic properties followed by $h$ and $l$ are: 
\begin{equation}\label{hxi}
h\xi=0,\,\, l\xi=0,\,\, \eta\circ h=0,\,\, Tr. h = Tr. \phi h = 0
\end{equation}
and
\begin{align}\label{hphi}
h\phi=-\phi h\,\,{\rm({\it i.e.},\,\, {\it h}\,\, anti-commutes\,\, with\,\, \phi).}
\end{align}

\noindent Also, $hX=\lambda X\implies h\phi X=-\lambda\phi X$, {\it i.e.}, if $\lambda$ is an eigen value of $h$ with the corresponding eigen vector $X$, then $-\lambda$ is also an eigen value of $h$ corresponding to the eigen vector $\phi X$. If $S$ denotes the Ricci operator and $\nabla$ denotes the Levi-Civita connection of the metric $g$, then using the above properties of $h$ and $l$, we can easily calculate the following formulas for a paracontact metric manifold $M^{2n+1}$:
\begin{equation}\label{nablax}
\nabla_{X}\xi = -\phi X + \phi hX,\,\, \nabla_{\xi}\xi=0.\end{equation}
\begin{equation}\label{nablaxi}
\nabla_{\xi}\phi=0.\end{equation}
\begin{equation}\label{Trl}
Trl=g(S\xi, \xi)=Trh^2-2n.\end{equation}
\begin{equation}\label{phil}
\phi l\phi + l= 2\left(h^{2}-\phi^{2}\right).\end{equation}
\begin{equation}\label{nablaxih}
\nabla_{\xi}h=-\phi-\phi l+h^{2}\phi.\end{equation}
If the Reeb vector field $\xi$ is Killing, {\emph i.e.,} equivalently $h=0$, then the paracontact metric manifold $M^{2n+1}$ is called a $K$-paracontact manifold \cite{gn}. On a $3$-dimensional pseudo-Riemannian manifold, since the conformal curvature tensor vanishes identically, therefore the curvature tensor $R$ takes the form \cite{irken} 
\begin{align}\label{rxyz}
R(X,Y)Z=&g(Y,Z)SX-g(X,Z)SY+g(SY,Z)X\nonumber\\
-&g(SX,Z)Y-\frac{r}{2}\left\{g(Y,Z)X-g(X,Z)Y\right\}
\end{align}
where $r$ is the scalar curvature of the manifold and the Ricci operator $S$ is defined by
\begin{equation}
g(SX,Y)=Ric(X,Y).
\end{equation}
where $Ric$ is the Ricci tensor.

\section{$k$-nullity distribution}
In contact geometry, the notion of $k$-nullity distribution is introduced by Tanno (1988, \cite{st}). The $k$-nullity distribution of a Riemannian manifold $(M, g)$, for a real number $k$, is a distribution
\begin{equation}\label{nk}
N(k):p\rightarrow N_{p}(k)=\left [ Z\in \Gamma(T_{p}M): R(X,Y)Z = k\left\{g(Y,Z)X-g(X,Z)Y\right\}\right]
\end{equation}
for any $X,\,Y\in \Gamma(T_{p}M)$, where $R$ and $\Gamma(T_{p}M)$ respectively denotes the curvature tensor and the tangent vector space of $M^{2n+1}$ at any point $p\in M$. If the characteristic vector field $\xi$ of a paracontact metric manifold belongs to the $k$-nullity distribution then the following relation holds \cite{st}
\begin{equation}\label{rxyxi}
R(X,Y)\xi=k(\eta(Y)X-\eta(X)Y).
\end{equation}
{\begin{definition}A paracontact metric manifold is said to be $\eta$-Einstein \cite{de203}, if the Ricci operator $S$ can be written in the following form:
\begin{equation}
S=aI+b\eta\otimes\xi,
\end{equation}
where $a$ and $b$ are some functions.\end{definition}
\noindent Now we will prove the following:
\begin{theorem}\label{trl}
Let $M^{3}(\phi, \xi, \eta, g)$ be a paracontact metric manifold with $\phi S=S\phi$. Then for any $X\in\Gamma(TM^3)$ orthogonal to $\xi$
\begin{enumerate}
\item[\rm (i)]the function $Trl$ vanishes on $M^{3}$ and

\item[\rm (ii)] the function $f$ defined by $lX=f X$ is constant everywhere on $M^{3}$.

\end{enumerate}

\end{theorem}
\begin{proof}By virtue of equations (\ref{eta}), (\ref{phixi}), (\ref{Trl}) and $\phi S=S\phi$, we have
\begin{equation}\label{Sxi}
S\xi=(Trl)\xi.\end{equation}
From equation (\ref{rxyz}), using the definition of $l$ and (\ref{Sxi}), we have for any $X$
\begin{align}lX=& R(X,\xi)\xi\nonumber\\=&g(\xi,\xi)SX-g(X,\xi)S\xi+g(S\xi,\xi)X\nonumber \\
-&g(SX,\xi)\xi-({r}/{2})\left(g(\xi,\xi)X-g(X,\xi)\xi\right)\nonumber
\end{align}
which gives
\begin{align}\label{2.2}
 lX=SX+\left(Trl-{r}/{2}\right)X+\eta(X)\left({r}/{2}-2Trl\right)\xi
\end{align}
since, $\eta(SX)\xi=\eta(X)(Trl)\xi$.
As a matter of fact if at a point $p\in M^{3}$ there exists $X\in \Gamma(T_{p}M^{3})$ such that $lX=0$ for $X\ne \xi$, then $l=0$ at $p$. So, let us assume that $l\ne0$ on a neighbourhood of a point $p$. From (\ref{phileq}) and (\ref{hphi}), we get
\begin{eqnarray}
g(\phi X, lX)=-g(X,\phi lX)=-g(\phi X, lX)\implies g(\phi X, lX)=0.\nonumber
\end{eqnarray}
Therefore, $lX$ is parallel to $X$ for any $X$ orthogonal to $\xi$. Since, $lX=fX$ for any $X$ orthogonal to $\xi$. Taking a local orthogonal frame $\left\{e_1, e_2, \xi\right\}$ where $-g(e_1,e_1)=g(e_2,e_2)=g(\xi,\xi)=1$ and $e_1,\,e_2$ are orthogonal to $\xi$. Then, by definition
\begin{align}\label{trl0}
&Trl=\sum_{i=1}^{3} g(le_i,e_i)=g(le_1,e_1)+g(le_2,e_2)+g(l\xi,\xi) \nonumber\\
&=fg(e_1,e_1)+fg(e_2,e_2)=0 
\end{align}
which proves (i). \\
For any $X$, we can write
\begin{equation}\label{2.6}
lX=f\phi^{2}X.
\end{equation}
From (\ref{2.6}), (\ref{nablaxi}) and (\ref{xitrl}), we have
\begin{align}\label{xif}
\xi f=0.
\end{align}
With the help of equations (\ref{2.6}) and (\ref{2.2}), we find
\begin{align}
lX=SX+\left(Trl-{r/}{2}\right)X+\eta(X)\left({r}/{2}-2Trl\right)\xi 
\nonumber
\end{align} 
which gives
\begin{align}\label{2.7}
SX=aX+b\eta(X)\xi
\end{align}
where 
\begin{eqnarray}\label{2.8}
a=(f-Trl+r/2),&& b=(2Trl-{r/}{2}-f).\end{eqnarray}
From the second identity of Bianchi, we get
\begin{equation} \label{2.11} 
\left(\nabla _{X} R\right)\left(Y,\xi ,Z\right)+\left(\nabla _{Y} R\right)\left(\xi ,X,Z\right)+\left(\nabla _{\xi } R\right)\left(X,Y,Z\right)=0. 
\end{equation} 
Employing (\ref{2.7}) in (\ref{rxyz}), we have
\begin{align} 
R\left(X,Y\right)Z= &g\left(Y,Z\right)\left(aX+b\eta \left(X\right)\xi \right)-g\left(X,Z\right)\left(aY+b\eta \left(Y\right)\xi \right)\nonumber \\
+& g\left(aY+b\eta \left(Y\right)\xi ,Z\right)X-g\left(aX+b\eta \left(X\right)\xi ,Z\right)Y\nonumber\\
-&({r}/{2}) \left(g\left(Y,Z\right)X-g\left(X,Z\right)Y\right) \nonumber \\
=& ag\left(Y,Z\right)X+b\eta \left(X\right)g\left(Y,Z\right)\xi -ag\left(X,Z\right)Y-b\eta \left(Y\right)g\left(X,Z\right)\xi \nonumber \\
+& ag\left(Y,Z\right)X+ b\eta \left(Y\right)g\left(\xi ,Z\right)X-ag\left(X,Z\right)Y-b\eta \left(X\right)g\left(\xi ,Z\right)Y  \nonumber \\
-&({r}/{2})(g\left(Y,Z\right)X-g\left(X,Z\right)Y) \nonumber \\
=& \left\{\left(2a-{r}/{2} \right)g\left(Y,Z\right)X+b\eta \left(Y\right)\eta \left(Z\right)X\right\} \nonumber \\
-&\left\{\left(2a-{r}/{2} \right)g\left(X,Z\right)Y+b\eta \left(X\right)\eta \left(Z\right)Y\right\} \nonumber \\
+& b\eta \left(X\right)g\left(Y,Z\right)\xi -b\eta \left(Y\right)g\left(X,Z\right)\xi \nonumber
\end{align} 
that is,
\begin{eqnarray} \label{2.12} 
R\left(X,Y\right)Z &=&\left\{\left(\gamma g\left(Y,Z\right)+ b\eta \left(Y\right)\eta \left(Z\right)\right)X-\left(\gamma g\left(X,Z\right)+b\eta \left(X\right)\eta \left(Z\right)\right)Y\right\} \nonumber\\ & +& b\left\{\eta \left(X\right)g\left(Y,Z\right)-\eta \left(Y\right)g\left(X,Z\right)\right\}\xi  
\end{eqnarray}
where $\gamma =2a-{r}/{2}$. From (\ref{2.12}) for $Z=\xi$, we get 
\begin{align}
R(X, Y)\xi =& \left\{\gamma \eta \left(Y\right)+b\eta \left(Y\right)\right\}X-\left\{\gamma \eta \left(X\right)+b\eta \left(X\right)\right\}Y\nonumber\\
+&b\left\{\eta \left(X\right)\eta \left(Y\right)-\eta \left(Y\right)\eta \left(X\right)\right\}\xi \nonumber \\ 
=& \left(\gamma +b\right)\left(\eta \left(Y\right)X-\eta \left(X\right)Y\right)\nonumber
\end{align} 
switching the values of $b$ and $\gamma$ in the above equation, we have
\begin{align}\label{2.13}
R(X,Y)\xi=f\left(\eta \left(Y\right)X-\eta \left(X\right)Y\right).
\end{align}
Using (\ref{2.13}), we obtain
\begin{align}R&\left(Y,\xi \right)\xi =f\left(Y-\eta \left(Y\right)\xi \right),\tag{A}
\\                                                                             
 R&\left(\nabla _{X} Y,\xi \right)\xi =f\left(\nabla _{X} Y-\eta \left(\nabla _{X} Y\right)\xi \right),\tag{B}
\\
R&\left(Y,\nabla _{X} \xi \right)\xi =f \left(\eta \left(\nabla _{X} \xi \right)Y-\eta \left(Y\right)\nabla _{X} \xi \right),\tag{C}
\\
 R&\left(Y,\xi \right)\nabla _{X} \xi =\left\{\gamma \eta \left(\nabla _{X} \xi \right)+b\eta \left(\nabla _{X} \xi \right)\right\}Y-\left\{\gamma g\left(Y,\nabla _{X} \xi \right)+b\eta \left(Y\right)\eta \left(\nabla _{X} \xi \right)\right\}\xi\nonumber
\\ &\hspace{2cm}+b\left\{\eta \left(Y\right)\eta \left(\nabla _{X} \xi \right)-g\left(Y,\nabla _{X} \xi \right)\right\}\xi. \tag{D}
\end{align}
From (A), we also have 
\begin{align}
&\left(\nabla _{X} R\right)\left(Y,\xi ,\xi \right)+ R\left(\nabla _{X} Y,\xi \right)\xi+R\left(Y,\nabla _{X} \xi \right)\xi+R\left(Y,\xi \right)\nabla _{X} \xi \nonumber \\
&=\left(Xf\right)\left(Y-\eta \left(Y\right)\xi \right)+f\bigl\{\nabla _{X} Y-\left(\left(\nabla _{X}\eta\right)Y\right)\xi-\eta\left(\nabla _{X} Y\right)\xi -\eta \left(Y\right)\nabla _{X}\xi \bigr\}.\tag{E}
\end{align}
\noindent Hence for $X,\,Y$ orthogonal to $\xi $, we get from (B), (C), (D) and (E)
\[\left(\nabla _{X} R\right)\left(Y,\xi ,\xi \right)+f \left(\nabla _{X} Y\right)=\left(Xf\right)Y+f\left(\nabla _{X} Y\right)\] 
this implies that
\begin{align}
\left(\nabla _{X} R\right)\left(Y,\xi ,\xi \right)=\left(Xf\right)Y. \tag{F}
\end{align}
Applying $Z=\xi $ in equation (\ref{2.11}), we get
\begin{align}\left(\nabla _{X} R\right)\left(Y,\xi ,\xi \right)+\left(\nabla _{Y} R\right)\left(\xi ,X,\xi \right)+\left(\nabla _{\xi } R\right)\left(X,Y,\xi \right)=0. \tag{G}
\end{align}
\noindent  From (\ref{2.13}), we can write
\begin{align}\label{2.14}
&{}\left(\nabla _{\xi } R\right) \left(X,Y,\xi \right)+ R\left(\nabla _{\xi } X,Y\right){\xi}+R\left(X,\nabla _{\xi } Y\right){\xi}+R\left(X,Y \right)\nabla _{\xi } \xi \nonumber\\
&=\left(\xi f\right)\bigl\{\eta \left(Y\right)X-\eta \left(X\right)Y\bigr\}+f\bigl\{((\nabla_{\xi}\eta)Y) X-((\nabla_{\xi}\eta)X)Y\bigr\}\nonumber \\ 
&+f \bigl\{\eta (\nabla _{\xi } Y)X+\eta (Y)\nabla _{\xi } X-\eta(\nabla _{\xi } X)Y-\eta(X)\nabla _{\xi } Y\bigr\}.
\end{align}
Using (\ref{nablax}) and (\ref{xif}) in (\ref{2.14}), we get
\begin{align}
&\left(\nabla _{\xi } R\right)\left(X,Y,\xi \right)+ R\left(\nabla _{\xi } X,Y,\xi \right)+R\left(X,\nabla _{\xi } Y,\xi \right)                                                                               \nonumber\\ 
&=f \bigl\{\eta \left(Y\right)\nabla _{\xi } X-\eta \left(\nabla _{\xi } X\right)Y +\eta \left(\nabla _{\xi } Y\right)X-\eta \left(X\right)\nabla_{\xi} Y\bigr\}\nonumber
\\
&+f\bigl\{\left(\left(\nabla _{\xi } \eta \right)Y\right)X-\left(\left(\nabla _{\xi } \eta \right)X\right)Y\bigr\} \nonumber
\\
&= R\left(\nabla _{\xi } X,Y,\xi \right)+R\left(X,\nabla _{\xi } Y,\xi \right) +f\left(g\left(Y,\nabla _{\xi } \xi \right)X-g\left(X,\nabla _{\xi } \xi \right)Y\right)\nonumber
\end{align}
that is
\begin{align}\left(\nabla _{\xi } R\right)\left(X,Y,\xi \right)=0.\tag{H}\end{align}

\noindent From (H) and (G), we find that $\left(\nabla _{X} R\right)\left(Y,\xi ,\xi \right)=\left(\nabla _{Y} R\right)\left(X,\xi ,\xi \right)$ and by the use of (F) this implies, $\left(Xf\right)Y= \left(Yf\right)X$, for $X$ and $Y$ orthogonal to $\xi $. Therefore $Xf=0$, for $X$ orthogonal to $\xi $, but $\xi f=0$, so the function $f$ is constant everywhere on $M^3$ and we reached at the end of the proof.
\end{proof}

\begin{prop}\label{lphi}
Let $M^3(\phi, \xi, \eta, g)$ be a paracontact metric manifold with $S\phi=\phi S$, then we have

\begin{enumerate}
\item[\rm (i)]\hspace{1cm} $l\phi=\phi l$.

\item[\rm (ii)]\hspace{1cm} $l=h^{2}-\phi ^{2}$.

\item[\rm (iii)]\hspace{1cm} $\xi\,Trl=0$.
\end{enumerate}
\end{prop}
\begin{proof}
 $S\phi=\phi S$ and (\ref{2.2}) yields
\begin{equation}\label{phileq}
\phi l=l\phi
\end{equation}
Further using (\ref{phileq}), (\ref{phil}) and (\ref{nablaxih}), we obtain that
\begin{equation}\label{l}
l=h^{2}-\phi^{2}
\end{equation}
and
\begin{eqnarray}
\nabla_{\xi}h &=&-\phi+h^{2}\phi-\phi l\nonumber\\ 
&=&-\phi +h^{2}\phi -\phi(h^{2}-\phi^{2})\nonumber\\ &=&-\phi + h^{2}\phi-\phi hh+\phi^{3} \nonumber \\&=& h^{2}\phi +h\phi h \nonumber
\end{eqnarray}
that is,
\begin{align}\label{2.5}
\nabla_{\xi}h=0.
\end{align}
We now differentiate equation (\ref{l}) with respect to $\xi$ and use equations (\ref{nablaxi}) and (\ref{2.5}), and find that 
\begin{align}\label{xitrl}
\nabla_{\xi}l=0 \implies \xi\,Trl=0.
\end{align}
This ends the proof.
\end{proof}

\begin{remark} Taking $ l=0$ everywhere, and using (\ref{rxyz}), (\ref{Sxi}) and (\ref{2.2}), we get \\$R\left(X,Y\right)\xi =0$. This together with the theorem 3.3 of \cite{szvt} gives that $M^{3}$ is flat.\end{remark}

\begin{theorem}\label{equiv} If $M^{3}(\phi ,\xi ,\eta ,g) $ be a paracontact metric manifold, then the following conditions are equivalent:

\begin{enumerate}
\item[\rm (i)]\hspace{1cm}  $M^{3} $ is $\eta$-Einstein.

\item[\rm (ii)] \hspace{1cm} $S\phi =\phi S$.

\item[\rm (iii)]\hspace{1cm} $\xi$ belongs to the $k$-nullity distribution.
\end{enumerate}
\end{theorem}
\begin{proof}

\noindent \textbf{(i)$\implies$(ii).}

\noindent Let $M^{3}$ be $\eta $-Einstein that is, $S=aI+b\eta \otimes \xi$. Therefore $S\phi =a\phi +b\left(\eta \circ \phi \right)\otimes \xi =a\phi$. Also, $\phi S=a\phi +b\eta \otimes \phi \xi =a\phi$. Hence $S\phi =\phi S$. 

\noindent \textbf{(ii)$\implies$(iii).}

\noindent Let $S\phi =\phi S$. Then, we have from (\ref{2.13}) \[R\left(X,Y\right)\xi =f\left(\eta \left(Y\right)X-\eta \left(X\right)Y\right).\]
By the theorem \ref{trl} we have, $f={\rm constan}t=k$ (say), therefore,
\[R\left(X,Y\right)\xi =k\left(\eta \left(Y\right)X-\eta \left(X\right)Y\right)\] this implies that $\xi$ belongs to the $k$-nullity distribution.

\noindent \textbf{(iii)$\implies$(i).}

\noindent Let $\xi$ belongs to the $k$-nullity distribution. Then,
\begin{equation} \label{2.16} 
R\left(X,Y\right)\xi =k\left(\eta \left(Y\right)X-\eta \left(X\right)Y\right) 
\end{equation} 
where $k$ is a constant.

\noindent Contracting (\ref{2.16}) with respect to $X$, we have
\[Ric\left(Y,\xi \right)=k\left(3\eta \left(Y\right)-\eta \left(Y\right)\right)=k\left(2\eta \left(Y\right)\right)=2k\eta \left(Y\right)\] 
that is,
\[{ S\xi =2k\xi } \] 
and so from (\ref{rxyz}), we find
\begin{equation} \label{2.17} 
R\left(X,Y\right)\xi =\eta \left(Y\right)SX-\eta \left(X\right)SY+\left(2k-{r}/{2} \right)\left(\eta \left(Y\right)X-\eta \left(X\right)Y\right). 
\end{equation} 
Comparing (\ref{2.16}) and (\ref{2.17}), we get
\[\eta \left(Y\right)\left\{SX+\left(k-{r}/{2} \right)X\right\}=\eta \left(X\right)\left\{SY+\left(k-{r}/{2} \right)Y\right\}.\] 
Taking $Y$ orthogonal to $\xi $ and $X=\xi $, we have
\[ SY=\left({r}/{2} -k\right)Y\] 
and so for any $Z$,
\[SZ=\left({r}/{2} -k\right)Z+\left(k-{r}/{2} \right)\eta \left(Z\right)\xi .\] 
This implies that $S=aI+b\eta \otimes \xi $, where $a={r}/{2} -k$ and $b=k-{r}/{2}. $ \\
Therefore $M^{3}$ is $\eta $-Einstein. This completes the proof of the theorem.\end{proof}
\medskip
\section{Torse forming vector fields}
\begin{definition}
A vector field $U$ defined by $g(X,U)=u(X)$ for any $X\in\Gamma (TM^3)$ is said to be torse forming vector field \cite{yk} (see also \cite{ma, js}) if
\begin{align}\label{torse}
(\nabla_{X}u)(Y)=sg(X,Y)+\alpha (X)u(Y),
\end{align}
where $s$ and $\alpha$ are called the {\it conformal scalar} and the {\it generating form} of $U$, respectively.

A torse forming vector field $U$ is called
\begin{itemize}
\item[$\bullet$] {\it recurrent} or {\it parallel}, if $s=0$,
\item[$\bullet$] {\it concircular}, if the generating form $\alpha$ is a gradient,
\item[$\bullet$] {\it convergent}, if it is concircular and $s={\rm const.\,exp}(\alpha)$.
\end{itemize}
\end{definition}

\begin{theorem}
Let $M^{3}(\phi ,\xi ,\eta ,g) $ be a paracontact metric manifold with $S\phi =\phi S$. Then the unit torse forming vector field in $M^3$ is concircular.
\end{theorem}
\begin{proof}
For a unit torse forming vector field $\widehat{U}$ corresponding to U, if we define $g(X,\widehat{U})=V(X)$, then
\begin{align}\label{5.2}
V(X)=u(X)/{\sqrt{u(U)}}. 
\end{align}
From (\ref{torse}) and (\ref{5.2}), we have
\begin{align}\label{5.3}
(\nabla_{X}V)(Y)=\mu g(X,Y)+\alpha(X)V(Y)
\end{align}
where $\mu=s/{\sqrt{u(U)}}$. Using $Y=\widehat{U}$ and $V(\widehat{U})=1$,  equation (\ref{5.3}) gives
\begin{align}\label{5.4}
\alpha(X)=-\mu V(X)
\end{align}
and hence (\ref{5.3}) can be expressed in the following form
\begin{align}\label{5.5}
(\nabla_{X}V)(Y)=\mu\bigl\{g(X,Y)-V(X)V(Y)\bigr\}
\end{align}
which shows that $V$ is closed. Now differentiating (\ref{5.5}) covariantly and using the Ricci identity, we obtain
\begin{align}\label{5.6}
V(R(X,Y)Z)=&(Y\mu)\bigl\{g(X,Z)-V(X)V(Z)\bigr\}-(X\mu)\bigl\{g(Y,Z)-V(Y)V(Z)\bigr\}\nonumber \\
+&\mu^{2}\bigl\{g(X,Z)V(Y)-g(Y,Z)V(X)\bigr\}.
\end{align}
By the use of theorem \ref{trl}, equations (\ref{2.13}), (\ref{gx}) and (\ref{5.2}), we have from (\ref{5.6})
\begin{align}
\bigl\{\eta(X)-\eta(\widehat{U})V(X)\bigr\}\bigl\{f+(\widehat{U}\mu) +\mu^2\bigr\}=0
\end{align}
which gives 
\begin{align}\label{I}
\bigl\{\eta(X)-\eta(\widehat{U})V(X)\bigr\}=0\tag{I}
\end{align}
or
\begin{align}\label{II}
\bigl\{f+(\widehat{U}\mu) +\mu^2\bigr\}=0.\tag{II}
\end{align}
If (\ref{I}) holds, then putting $X=\xi$ in (\ref{I}), we have $\eta (\widehat{U})=\pm 1$. This implies that
\begin{align}\label{5.7}
\eta(X)=\pm V(X)
\end{align}
From (\ref{nablax}), (\ref{5.5}) and (\ref{5.7}), we have $\mu=\pm A$(constant). Therefore $\alpha (X)=\mp A V(X)$. Hence the vector field $\widehat{U}$ is concircular.\\
If (\ref{II}) holds, then $\bigl\{\eta(X)-\eta(\widehat{U})V(X)\bigr\}\ne 0$. From (\ref{5.6}), we have
\begin{align}\label{5.8}
-g(SX, \widehat{U})=(X\mu)+(\widehat{U}\mu)V(X)+2\mu^{2}V(X).
\end{align}
Putting $X=\xi$ and using (\ref{Sxi}) in (\ref{5.8}), we have
\begin{align}\label{5.9}
\xi\mu=-(\mu^{2} + f)\eta(\widehat{U}).
\end{align}
In view of (\ref{5.9}) and $V(\xi)=\eta(\widehat{U})$ equation (\ref{5.6}) yields for $Y=Z=\xi$,
\begin{align}\label{5.10}
X\mu=-(\mu^{2} +f)V(X).
\end{align}
From (\ref{5.4}) and (\ref{5.10}), we get
\begin{align}\label{5.11}
Y\alpha(X)=(\mu^{2} +f)V(X)V(Y)-\mu(YV(X)).\tag{a}
\end{align}
We can also obtain
\begin{align}\label{5.12}
X\alpha(Y)=(\mu^{2} +f)V(Y)V(X)-\mu(XV(Y))\tag{b}
\end{align}
and
\begin{align}\label{5.13}
\alpha([X,Y])=-\mu V([X,Y]).
\end{align}
From (\ref{5.11}), (\ref{5.12}) and (\ref{5.13}), we have
\begin{align}
d\alpha(X,Y)=-\mu ((dV)(X,Y))\nonumber
\end{align}
which implies that  $\widehat{U}$ is concircular. These completes the proof of the theorem.
\end{proof}
\acknowledgement The authors would like to express their gratitude to D. E. Blair and R. Sharma for helpful comments and remarks in preparing this article.

\end{document}